\documentclass[reqno]{amsart}
\usepackage[scale=0.75, centering, headheight=14pt]{geometry}
\usepackage[T1]{fontenc}
\usepackage{lmodern}
\usepackage[english]{babel}

\usepackage{tikz}
\usepackage{bbm}
\usepackage{amsmath,amssymb,amsfonts,amsthm}
\usepackage{mathtools,accents}
\usepackage{mathrsfs}
\usepackage{xfrac}
\usepackage{array} 
\usepackage{aliascnt}
\usepackage{booktabs} 
\usepackage{array} 

\usepackage{verbatim} 
\usepackage{subfig} 

\usepackage{mathrsfs, dsfont}
\usepackage{amssymb}
\usepackage{amsthm}
\usepackage{amsmath,amsfonts,amssymb,esint}
\usepackage{graphics,color}
\usepackage{enumerate}
\usepackage{mathtools,centernot}

\usepackage{microtype}
\usepackage{paralist} 
\usepackage{cases}
\usepackage[initials]{amsrefs}
\allowdisplaybreaks

\usepackage{braket}
\usepackage{bm}

\usepackage[citecolor=blue,colorlinks]{hyperref}
\addto\extrasenglish{}

\usepackage{enumerate}
\usepackage{xcolor}
\usepackage{aliascnt}
\numberwithin{equation}{section}

\makeatletter
\def\newaliasedtheorem#1[#2]#3{
  \newaliascnt{#1@alt}{#2}
  \newtheorem{#1}[#1@alt]{#3}
  \expandafter\newcommand\csname #1@altname\endcsname{#3}
}
\makeatother

\def\avint{\mathop{\mathchoice{\,\rlap{-}\!\!\int}
                              {\rlap{\raise.15em{\scriptstyle -}}\kern-.2em\int}
                              {\rlap{\raise.09em{\scriptscriptstyle -}}\!\int}
                              {\rlap{-}\!\int}}\nolimits}

\def\avint{\mathop{\,\rlap{-}\!\!\int}\nolimits}

\newcommand{\A}{\mathcal{A}}

\DeclareMathOperator{\id}{id_2}

\newcommand{\R}{\mathbb{R}}
\newcommand{\E}{\mathds{E}}

\newcommand{\eps}{\varepsilon}

\DeclareMathOperator{\cof}{cof}
\DeclareMathOperator{\rank}{rank}

\DeclareMathOperator{\dv}{div}

\DeclareMathOperator{\loc}{loc}

\DeclareMathOperator{\idd}{id}
\DeclareMathOperator{\Lip}{Lip}

\theoremstyle{plain}

\newtheorem{Teo}{Theorem}[section]
\newtheorem{lemma}[Teo]{Lemma}
\newtheorem{prop}[Teo]{Proposition}

\theoremstyle{definition}
\newtheorem{Def}[Teo]{Definition}
\theoremstyle{remark}
\newtheorem{rem}[Teo]{Remark}
\newtheorem{Rem}[Teo]{Remarks}

\title[Lawson-Osserman Conjecture]{On the Lawson-Osserman conjecture}

\author[J. Hirsch]{Jonas Hirsch}
\address{Jonas Hirsch
\hfill\break Universit\"at Leipzig, Mathematisches Institut, Augustusplatz 10, 04109 Leipzig, Germany}
\email{hirsch.jonas@math.uni-leipzig.de}

\author[C. Mooney]{Connor Mooney}
\address{Connor Mooney
\hfill\break Department of Mathematics, UC Irvine}
\email{mooneycr@math.uci.edu}

\author[R. Tione]{Riccardo Tione}
\address{Riccardo Tione
\hfill\break Max Planck Institute for Mathematics in the Sciences, Inselstrasse 22, 04103 Leipzig, Germany}
\email{riccardo.tione@mis.mpg.de}

\begin{document}

\maketitle

\begin{abstract}
We prove that if $u : B_1 \subset \mathbb{R}^2 \rightarrow \mathbb{R}^n$ is a Lipschitz critical point of the area functional with respect to outer variations, then $u$ is smooth. This solves a conjecture of Lawson and Osserman from 1977 in the planar case.
\end{abstract}
\par
\medskip\noindent
{\sc MSC (2020): 35B65, 49Q05, 53A10, 58E12.
\par
}

\section{Introduction}

The $m$-dimensional area of the graph $\{(x,\,u(x)): x \in B_1\} \subset \mathbb{R}^{m+n}$ of a Lipschitz map $u : B_1 \subset \mathbb{R}^m \rightarrow \mathbb{R}^n$ is given by
\begin{equation}\label{AreaFunctional}
\E(u) \doteq \int_{B_1}  \sqrt{\det g}\,dx,
\end{equation}
where
$$g = \text{id}_m + Du^TDu$$
is the metric induced by the non-parametric map $x \mapsto (x,\,u(x))$. If $u$ is a critical point of $\E$ with respect to outer variations, that is,
\begin{equation}\label{OuterCrit}
\partial_i(\sqrt{\det g}g^{ij}\partial_ju^{l}) = 0,\, \quad l = 1, \dots , n,
\end{equation}
we say that $u$ is a Lipschitz weak solution of the minimal surface system, or that $u$ is outer critical. If $u$ is a critical point with respect to inner variations, that is,
\begin{equation}\label{InnerCrit}
\partial_i(\sqrt{\det g}g^{ij}) = 0, \quad j = 1, \dots , m,
\end{equation}
we say that $u$ is inner critical. Finally, we say that $u$ is stationary if it is both outer and inner critical. 
\\
\\
In the case that $u \in C^2$, it is easy to show that outer critical and stationary are equivalent. In \cite[Conjecture 2.1]{LOA}, Lawson and Osserman conjectured that the same is true when $u$ is just Lipschitz. In the codimension one case $n = 1$ this follows from the De Giorgi-Nash-Moser theorem \cite{DEG,NASH,NASHPRE,MOS}. The purpose of this paper is to confirm the Lawson-Osserman conjecture when $m = 2$ and the codimension $n$ is arbitrary:

\begin{Teo}\label{main}
Lipschitz weak solutions to the minimal surface system in dimension $m = 2$ are smooth.
\end{Teo}

\begin{rem}
One may draw an analogy between Theorem \ref{main} and the theory of harmonic maps. Indeed, deep work of H\'{e}lein \cite{HELBO} and Rivi\`{e}re \cite{RIV} shows that weakly harmonic maps (outer critical for the Dirichlet energy) are smooth when the domain is two-dimensional. In contrast, when the domain has dimension $m \geq 3$, there exist weakly harmonic maps that are \emph{wild}, e.g. nowhere continuous \cite{RIVAC}. In view of these results it is tempting to conjecture the existence of wild (e.g. nowhere $C^1$) Lipschitz outer critical maps for the area functional in higher dimension. However, the methods used to treat harmonic maps are quite different, and it is not clear what to expect for the area functional when $m \geq 3$ and $n \geq 2$.
\end{rem}

A major advantage of working with stationary maps is that the monotonicity formula is available, which reduces the analysis of possible singularities to the study of stationary cones (one-homogeneous stationary maps). It is known for example that stationary maps are smooth in dimension $m \leq 3$, see \cite{DFC,BAR}, and more generally that the Hausdorff dimension of the singular set is at most $m - 4$ \cite{DIMLER}. The latter result is sharp in view of the Lawson-Osserman cone constructed in \cite[Theorem 7.1]{LOA}. In contrast, much less is known about outer critical maps. 

\begin{rem}
The same questions can be asked for functionals associated to general strictly polyconvex integrands $F$, of which the area integrand $\sqrt{\det g}$ is an important example. For such polyconvex integrands $F$, Lipschitz outer critical maps can be nowhere $C^1$, even in dimension $m = n = 2$, \cite{LSP} (see \cite{SMVS} for related examples in the quasiconvex setting). The key to such constructions is the existence of certain matrix configurations related to $F$ known as $T_N$ configurations, and a technique known as convex integration. It turns out that these configurations are ruled out for maps that are stationary with respect to such functionals, \cite{DLDPKT,HRT}.
\end{rem}

\begin{rem}
Lipschitz inner critical maps for the area functional can be nowhere $C^1$, even in dimension $m = n = 2$, \cite{DUSE}.
\end{rem}

Our approach to Theorem \ref{main} combines ideas from complex analysis, the analysis of quasiconformal inner critical maps, the regularity theory for the Monge-Amp\`{e}re equation, and the codimension one case. Here we describe our approach in the case $m = n = 2$, for simplicity. In the first step we show using isothermal coordinates that $u$ can be written
$$u = v \circ \phi,$$
where $\phi$ is a quasiconformal map and $v$ is harmonic (this works in general codimension). The set $\{|Dv| = 0\}$ consists of isolated points. Using removability results for isolated points we may assume that $|Dv| \sim 1$. Near a point $x$ where $|\det Dv(\phi(x))| > 0$ we have that $u$ is a quasiconformal map (up to a rigid motion). The key in this situation is the observation that the inverse $w$ of $u$ is an inner critical quasiconformal map near $u(x)$. Using a connection between inner critical maps and the Monge-Amp\`{e}re equation we show that inner criticality is an elliptic condition on quasiconformal matrices, and as a result we establish the smoothness of $w$ near $u(x)$, thus the smoothness of $u$ near $x$. The alternative is that $|\det Dv(\phi(x))| = 0$. We can then establish using the quasiconformality of $\phi$ that $|\det Du|$ is small near $x$. Using that the area functional is convex near rank-one matrices, we can then prove the smoothness of $u$ near $x$ in the latter case as well.
\\
\\
In the general codimension case, the argument is the same at points where $Dv(\phi(x))$ has rank one (that is, we leverage the convexity of the area functional in this region). Near points where $Dv$ has rank two, the key observation is that the image of $u$ is contained in a smooth two-dimensional surface in $\mathbb{R}^n$. This observation makes the analysis in general codimension similar to the case $n = 2$.

\begin{rem}
Each step in the proof of Theorem \ref{main} uses the two-dimensionality of the domain in a crucial way, and it is not clear what to expect in higher dimensions. A natural question is whether there exist $T_N$ configurations in the differential inclusions associated to outer variations in higher dimensions, which would be evidence towards the existence of wild solutions.
\end{rem}

The paper is organized as follows. In Section \ref{PRELNOT} we discuss preliminaries related to the area functional and quasiconformal mappings. In Section \ref{s:MTS} we reduce the proof of Theorem \ref{main} to justifying the use of isothermal coordinates and passing from outer variations for $u$ to inner variations of its inverse, the regularity of outer critical maps with small Jacobian determinant, and the regularity of quasiconformal inner critical maps. Sections \ref{s:OUTTO}-\ref{s:SMALLDET}-\ref{s:SOLIN} are devoted to the proofs of these results, respectively. Finally, the proofs of some technical lemmas are included in the Appendix.

\section*{Acknowledgments}
C. Mooney was supported by a Sloan Research Fellowship and NSF CAREER Grant DMS-2143668. He is grateful to Richard Schoen for bringing this problem to his attention.

\section{Preliminaries and notation}\label{PRELNOT}
\subsection{Notation}
We denote by $X^T$ the transpose of $X \in \R^{n \times m}$. If $n = m$ and $X$ is invertible, we denote by $X^{-T}$ the inverse of $X^T$. The standard Euclidean product on vectors $a,b \in \R^n$ is denoted by $(a,b)$, while the standard Euclidean product on matrices $A,B \in \R^{n\times m}$ is denoted by $\langle A,B \rangle$. The associated Euclidean norm is denoted both for vectors $a$ and matrices $A$ as $|a|$ and $|A|$.
\subsection{The area functional}
Here we discuss the structure of the area functional. We restrict our attention to the case that the domain has dimension $m = 2$, since that is the relevant case for our main result.
\\
\\
As recalled in the introduction, the stationarity condition with respect to the area functional for a map $u \in \Lip(B_1,\R^n)$ reads as:
\begin{equation}\label{sys0}
\begin{cases}
\partial_i(\sqrt{\det g}g^{ij}\partial_ju^{l}) = 0,& \quad l = 1,\dots, n,\\
\partial_i(\sqrt{\det g}g^{ij}) = 0,& \quad j = 1, \dots, 2,
\end{cases}
\end{equation}
where $(g^{ij}) = (g_{ij})^{-1}$, $g_{ij} =\delta_{ij} + (D u_i,D u_j)$. Let
\[
\E(u) \doteq \int_{B_1}\A(Du)dx, \quad u \in \Lip(B_1,\R^n),
\]
be the energy associated to the area integrand, see \eqref{eq:expA} below. The first line of \eqref{sys0} corresponds to assuming criticality of the graph with respect to \emph{outer} variations of the functional, i.e. 
\begin{equation}\label{outer}
\frac{d}{d\varepsilon}|_{\varepsilon = 0}\E(u + \varepsilon v) = 0, \quad \forall v \in C^\infty_c(B_1,\R^n)
\end{equation}
and the second line of \eqref{sys0} corresponds to assuming criticality of the graph with respect to \emph{inner} variations of the functional, i.e. 
\begin{equation}\label{inner}
\frac{d}{d\varepsilon}|_{\varepsilon = 0}\E(u\circ(x + \varepsilon\Phi)) = 0, \quad \forall \Phi \in C^\infty_c(B_1,\R^2).
\end{equation}
The area integrand on graphs $\mathcal{A}:\R^{n\times 2} \to \R$ reads as 
\[
\A(Z) \doteq \sqrt{\det(g)} = \sqrt{\det(M(Z)^TM(Z))},\quad\text{ where } \quad
M(Z)\doteq
\left (
\begin{array}{c}
\id \\
 Z
\end{array}\right), \text{ and $Z \in \R^{n\times 2}$}.
\]
More explicitly, for any $Z \in \R^{n\times 2}$,
\begin{equation}\label{eq:expA}
\begin{split}
\A(Z) &= \sqrt{1 + |Z|^2 + \det(Z^TZ)} \\
&= \sqrt{1 + |Z|^2 + |Z^1|^2|Z^2|^2 - (Z^1,Z^2)^2} \\
&= \sqrt{1 + |Z|^2 + \sum_{1\le a\le b\le n}\det(Z^{ab})^2},
\end{split}
\end{equation}
where, here and in the following, $Z^i$ is the $i$-th column of $Z$ and $Z^{ab}$ is the $2\times 2$ submatrix obtained from $Z$ considering just the $a$-th and the $b$-th rows. Denote with $\cof(M)$ the matrix for which $M\cof(M) = \det(M)\id$, for every $M \in \R^{2\times 2}$. Explicitly, if $$M = \left(\begin{array}{cc}a & b\\ c & d\end{array}\right), \;\text{ then } \cof(M) = \left(\begin{array}{cc}d & -b\\ -c & a\end{array}\right).$$ We compute
\begin{equation}\label{eq:expDA}
D\A(Z) = \frac{Z + \sum_{1\le a \le b \le n}\det(Z^{ab})C_{ab}(Z)}{\A(Z)},
\end{equation}
where $C_{ab}(Z)$ denotes the $n\times 2$ matrix defined as 
\begin{equation*}
(C_{ab}(Z))_{ij} =
\begin{cases}
0, \text{ if } i\neq a \text{ or } i \neq b\\
(\cof(Z^{ab})^T)_{ij}, \text{ otherwise.}
\end{cases}
\end{equation*}
Notice in particular that, due to \eqref{outer}, the first line of \eqref{sys0} intended as a system of $n$ equations corresponds to
\[
\dv(D\A(Du)) = 0.
\]
By direct computations one verifies that the Euler-Lagrange equations corresponding to \eqref{inner}, i.e. the second line of \eqref{sys0} intended as a system of two equations, corresponds to
\[
\dv(B(Du)) = 0,
\]
where
\begin{equation}\label{initial}
B(Z) = \A(Z)\id - Z^TD\A(Z)  = \frac{(1 + |Z|^2)\id - Z^TZ}{\A(Z)}.
\end{equation}
As a note, let us remark that \eqref{sys0} yields the following equalities, that link $D\A,B$ with the notation of the classical minimal graph system:
\begin{equation}\label{Basg}
B(Z) = \sqrt{\det(g)}g^{-1}
\end{equation}
and
\begin{equation}\label{shortexpDA}
 D\A(Z) = ZB(Z).
\end{equation}
It will be helpful to have an explicit expression for $B$:
\begin{equation}\label{eq:Bexp}
B(Z) =
\frac{1}{\A(Z)}\left(
\begin{array}{cc}
1+ |Z^2|^2 &-(Z^1,Z^2) \\
 -(Z^1,Z^2)& 1 +|Z^1|^2
\end{array}
\right).
\end{equation}
We need a further, very useful, short-hand notation. For matrices $X \in \R^{2\times 2}$, $Y \in \R^{(n-2)\times 2}$,
\begin{equation}\label{X|Y}
\A(X|Y), D\A(X|Y), B(X|Y) \text{ and } DB(X|Y)
\end{equation}
are simply $\A(Z)$, $D\A(Z)$, $B(Z)$ and $DB(Z)$ for $Z \in \R^{n\times 2}$ the matrix defined as $X$ in the first two rows, and $Y$ in the bottom $n-2$ rows. If $n = 2$, of course this notation is meaningless and one only considers $\A(X),D\A(X), B(X)$ and $DB(X)$.
\\
\\
We conclude this section by collecting a few facts of technical nature which will prove to be useful in the paper. Since the proof of this Lemma is straighforward but rather lengthy, we postpone it to Appendix \ref{sec:PLA}.

\begin{lemma}\label{lem:LA}
The following hold, for all $Z \in \R^{n\times 2}, X \in \R^{2\times 2}, Y \in \R^{(n-2)\times 2}$:
\begin{enumerate}[(i)]
\item $B(NZ) = B(Z)$ for every $N \in O(n) \doteq \{Y: Y^TY = \idd_n\}$;\label{orthr}
\item $B(ZM) = M^TB(Z)M$ for every $M \in O(2)$;\label{orth}
\item $D\A(NZM) = ND\A(Z)M$, for every $N \in O(n),M \in O(2)$; \label{DAorth}
\item $B(Z)$ is symmetric, $B(Z)_{11} >0, B(Z)_{22} > 0$, $\det(B(Z)) = 1$;\label{signs}
\item If $\det(X) > 0$, then 
\[
B(X^{-1}|Y) = \frac{XB(X|YX)X^T}{\det(X)};
\]\label{xoxo}
\item If $S = \sqrt{\id + Y^TY}$, then $$B(X|Y) =  \det(S)S^{-1}B(XS^{-1}|0)S^{-1};$$\label{BsimpS}
\item For every $K,L > 0$, if $|X|^2 \le K\det(X)$ and $|Y| \le L$, then there exists $C = C(n,K,L) > 0$ such that $$|B(X|Y)| \le C.$$ \label{Bdd}
\item For every $K,L > 0$, if $|X|^2 \le K\det(X)$ and $|Y| \le L$, then there exists $C = C(n,K,L) > 0$ such that $$|DB(X|Y)| \le C(1+ |X|)^{-1}.$$ \label{gradBdd}
\end{enumerate}
\end{lemma}

\subsection{Quasiconformal mappings}
We will need a few well-known facts about quasiregular and quasiconformal mappings.
\begin{Def}
Let $\Omega \subset \R^2$ be open. We say that $\varphi\in W^{1,2}(\Omega,\R^2)$ is quasiregular if there exists $K \ge 1$ such that for a.e. $x \in \Omega$
\[
|D\varphi|^2(x) \le K\det(D\varphi(x)).
\]
We say that it is quasiconformal if, in addition, $\varphi$ is a homeomorphism.
\end{Def}

\begin{Teo}\label{thm:quas}
Let $\varphi \in W^{1,2}(\Omega_1,\Omega_2), \psi \in W^{1,2}(\Omega_2,\Omega_3)$ be quasiconformal maps. Then:
\begin{enumerate}
\item\label{quasma1} $\varphi^{-1}: \Omega_2 \to \Omega_1$ is a $W^{1,2}(\Omega_2,\Omega_1)$ quasiconformal mapping, with gradient
\[
D(\varphi^{-1}(y)) = (D\varphi)^{-1}(\varphi^{-1}(y)), \text{ a.e. in } \Omega_2;
\]
\item\label{quasma2} the chain rule holds, i.e. for every $f \in \Lip(\Omega_2)$, $f\circ \varphi \in W^{1,2}(\Omega_1)$ with
\[
D(f\circ\varphi)(y) = Df(\varphi(y))D\varphi(y), \text{ a.e. in }\Omega_1;
\]
\item\label{quasma3} the composition $\psi \circ \varphi$ is still quasiconformal.
\end{enumerate}
\end{Teo}
\begin{proof}
\eqref{quasma1} can be found in \cite[Theorem 9.1]{BI}, while the chain rule \eqref{quasma2} is the content of \cite[Lemma 9.6]{BI}. The composition result \eqref{quasma3} is proved in \cite[Theorem 9.4]{BI}.
\end{proof}

\section{Main Theorem and strategy}\label{s:MTS}

To begin we state our main theorem more precisely, using the terminology defined in the previous section:

\begin{Teo}\label{MT}
Let $B_1 \subset \R^2$ and $u \in \Lip(B_1,\R^n)$. Suppose $u$ solves
\[
\dv(D\A(Du)) = 0 \text{ in the sense of distributions on $B_1$}.
\]
Then, $u$ is smooth. In particular, $u$ solves $\dv(B(Du)) = 0$ in $B_1$ and hence it is a solution to \eqref{sys0}.
\end{Teo}
The proof of Theorem \ref{MT} occupies the rest of this paper. Since this proof is made of various pieces, it is better to give an overview of the strategy here and use the remaining sections to show our claims. This overview constitutes the proof of Proposition \ref{simpMT1} below. We start by reducing Theorem \ref{MT} to showing a milder statement.
\begin{lemma}\label{simpMT}
Let $B_1 \subset \R^2$ and $u \in \Lip(B_1,\R^n)$. Suppose $u$ solves
\[
\dv(D\A(Du)) = 0 \text{ in the sense of distributions on $B_1$}.
\]
If $u \in C^{1,\alpha}$ for some $\alpha > 0$ in $B_1$ up to a discrete set $S$, then $u$ is smooth in $B_1$ and Theorem \ref{MT} holds.
\end{lemma}
\begin{proof}
The $C^{1,\alpha}$ regularity can be bootstrapped in a standard way to smoothness in $B_1\setminus S$, since the area functional satisfies the Legendre-Hadamard condition on compact sets \cite[Lemma 6.7]{RT}. If $u \in C^\infty(B_1\setminus S, \R^2)$, then $\dv(D\A(Du)) = 0$ in $B_1 \setminus S$ implies that $\dv (B(Du)) = 0$. This holds in much higher generality: if $u$ is a $C^2$ solution to $\dv(Df(Du)) = 0$ in an open set $\Omega$, then the inner variation equations $\dv(Du^TDf(Du)-f(Du)\id) = 0$ hold as well in $\Omega$, as can be checked by direct computation simply assuming that $f: \R^{n\times 2} \to \R$ is $C^1$. Thus, $u$ is a smooth solution to \eqref{sys0} in $B_1\setminus S$. Since $u$ is Lipschitz in the whole $B_1$, the isolated singularities coming from the set $S$ can be removed and it can be easily verified that $u$ is a Lipschitz solution to \eqref{sys0} in $B_1$. Thus, $u$ is smooth in $B_1$, see for instance \cite[Proposition 4.3]{RT}.
\end{proof}
Lemma \ref{simpMT} allows us to reduce to proving the following:
\begin{prop}\label{simpMT1}
Let $B_1 \subset \R^2$ and $u \in \Lip(B_1,\R^n)$. Suppose $u$ solves
\[
\dv(D\A(Du)) = 0 \text{ in the sense of distributions on $B_1$}.
\]
Then $u \in C^{1,\alpha}$ for some $\alpha > 0$ in $B_1$ up to a discrete set $S$.
\end{prop}
\begin{proof}
In Lemma \ref{lem:qc}, we will show the factorization of $u = v \circ \phi$, where $\phi: B_1 \to \phi(B_1)$ is quasiconformal and $v: \phi(B_1) \to \R^n$ is harmonic. This is reminescent of the analogous transformation in the foundational works of J. Douglas \cite{DOUG} and T. Rado \cite{RADO}. Define $U \doteq \phi(B_1)$, which is open since $\phi$ is a homeomorphism. Consider now the sets
\begin{equation}\label{E1}
E_1\doteq \{x \in U: |Dv|(x) = 0\},\quad \Omega \doteq U \setminus E_1.
\end{equation}
Since $v$ is harmonic and $U \subset \R^2$, either $E_1 = U$ or $E_1$ is discrete. Indeed, let $v_1$ be the first component of $v$. The set $\{z \in U: Dv_1(z) = 0\}$ is discrete, since locally $v_1$ is the real part of a holomorphic function $f = f(z)$, and $Dv_1(z) = 0$ implies $f'(z) = 0$. This shows that $E_1$ is either the whole $U$ or it is discrete, as wanted. We will only treat the second case, the first being trivial. Since $\phi$ is a homeomorphism, $S \doteq \phi^{-1}(E_1)$ is again discrete. Recalling the notation introduced in \eqref{eq:expA}, we can split
\[
\Omega = Z \cup O = \{y \in \Omega: \forall a,b, \det(Dv^{ab})(y) =0\}\cup \{y \in \Omega: \exists a,b \text{ such that } |\det(Dv^{ab})|(y) > 0\}.
\]
Of course, $Z$ is closed and $O$ is open relatively to $\Omega$. We will show separately that, if $\phi(x_0) \in Z$ or $\phi(x_0) \in O$, then there exists $r > 0$ such that $u \in C^{1,\alpha}(B_{r}(x_0),\R^n)$. This will conclude the proof. Let us analyze the two situations separately.
\\
\\
\fbox{Case 1: $\phi(x_0) \in Z$.} At $x_0$ we have $\det(Dv^{ab})(\phi(x_0)) = 0$ for all $1\le a < b \le n$ by definition of $Z$. Furthermore, since $x_0 \in \phi^{-1}(\Omega)$, $|Dv|(\phi(x_0)) \neq 0$. Hence, there exists a ball $B \doteq B_r(x_0)$ on which $\|\det(Dv^{ab})\circ\phi\|_{L^\infty(B)} \le\eps_1, \forall 1 \le a < b \le n$, and $|Dv|(\phi(x)) \ge c > 0$ for all $x \in B$. Here, $\eps_1 > 0$ is a small constant which will be chosen later. We start by noticing that $\phi$ is Lipschitz in $B_r(x_0)$. Indeed, since $\phi$ is quasiconformal, we can exploit Theorem \ref{thm:quas} to see that for almost all $x \in B_1$
\begin{equation}\label{philip}
k|Dv\circ \phi|(x)|D\phi|(x) \le |Dv\circ\phi(x)D\phi(x)| =|Du(x)| \le \Lip(u),
\end{equation}
where $k>0$ only depends on the quasiconformality constant of $\phi$. Combining this with the fact that $|Dv|(\phi(x)) \ge c$, we readily obtain that $\phi \in \Lip(B_r(x_0))$. Therefore, for all $x \in B_r(x_0)$ and for all $1 \le a < b \le n$,
\begin{equation}\label{CEPS}
|\det(Du^{ab})|(x) = |\det(Dv^{ab})|\circ\phi(x)|\det(D\phi)|(x) \le C\eps_1,
\end{equation}
for some positive constant $C$ only depending on $|Dv|\circ \phi(x_0)$, the quasiconformality constant $K$ of $\phi$ and the Lipschitz constant $\Lambda$ of $u$. Proposition \ref{prop:regsmalldet} shows that, given $\Lambda > 0$, there exists $\eps_2 = \eps_2(\Lambda) >0$ such that every solution $u$ to $\dv(D\A(Du)) = 0$ with $\|Du\|_{L^\infty(B_r(x_0))}\le \Lambda$ and $\|\det(Du^{ab})\|_{L^\infty(B_r(x_0))} \le \eps_2$ for all $1\le a < b \le n$ belongs to $C^{1,\alpha}$ in $B_{\frac{r}{2}}(x_0)$. This Proposition allows us to fix $\eps_1 \doteq \frac{\eps_2}{C}$, where $C$ is the constant appearing in \eqref{CEPS} and find $r = r(x_0,K,\Lambda)$ consequently.
\\
\\
\fbox{Case 2: $\phi(x_0) \in O$.} By definition of $O$, we can find a ball $D \doteq B_r(x_0)$ and $1 \le a < b \le n$ on which $|\det(Dv^{ab})|(\phi(x_0)) \ge c > 0$. We can, in particular, assume that 
\begin{equation}\label{vdiffeo}
\left(\begin{array}{c}v^a\\v^b\end{array}\right) \text{ is a diffeomorphism on $\phi(D)$.}
\end{equation}
Assume for the moment $a = 1, b = 2$. By continuity $\det(Dv^{12})\circ\phi$ has a sign on $D$, which we assume for the moment to be positive. It is convenient to split the maps under consideration in components: let $f: B_1 \to \R^2, h:B_1 \to \R^{n-2}, g: U \to \R^2, \tilde v: U \to \R^{n-2}$ be such that
\begin{equation}\label{splitu}
u = \left(\begin{array}{cc} f \\ h\end{array}\right) \text{ and } v = \left(\begin{array}{cc} g \\ \tilde v\end{array}\right).
\end{equation}
By \eqref{vdiffeo}, $g$ is a quasiconformal map on $\phi(D)$ and hence the composition $g\circ\phi = f$ is quasiconformal on $D$, according to Theorem \ref{thm:quas}. According to the same theorem, the inverse $w \doteq f^{-1}$ is quasiconformal as well. In addition, let us set
\begin{equation}\label{splitpsi}
\psi \doteq \tilde v \circ g^{-1} \in C^\infty(f(D), \R^{n-2}).
\end{equation}
In Lemma \ref{lem:outtoin}, we will show that
\begin{equation}\label{Bredu}
\dv(B(Dw|D\psi)) = 0 \text{ in the sense of distributions in $E \doteq f(D)$},
\end{equation}
where we used notation \eqref{X|Y}. Theorem \ref{teo:regin} shows that a quasiconformal solution to \eqref{Bredu} with Jacobian bounded below belongs to $C^{1,\alpha}(E',\R^2)$, where $E' \doteq f(B_{\frac{r}{2}}(x_0))$. To see that $\det(Dw)$ is bounded below, simply write
\begin{equation}\label{detDwbdd}
\det(Dw)(z) = \frac{1}{\det(Df)\circ w(z)}, \text{ for all $z \in E$}
\end{equation}
and use the fact that $u$, and hence $f$, is Lipschitz. Therefore the classical Inverse Function Theorem yields $f \in C^{1,\alpha}(B_{\frac{r}{2}}(x_0))$. Furthermore
\begin{equation}\label{splitu2}
h = \tilde v \circ \phi = \tilde v \circ g^{-1} \circ g\circ \phi =  \psi \circ f,
\end{equation}
and hence $h$ is as smooth as $f$ is. This implies the same regularity on $u$, and concludes the proof. In case $a \neq 1$, $b \neq 2$, or $\det(Dv^{ab})\circ \phi$ has the undesired sign on $D$, then we could simply employ the above reasoning on
\[
\tilde u \doteq Mu,
\]
where $M \in O(n)$ is the matrix defined as follows. First, consider $M_1 \in O(n)$ to be the matrix that swaps the $a$-th and $b$-th components of an $\R^n$ vector with the first and second components respectively. Next, let $M_2 \in O(n)$ be the matrix that acts as
\[
M_2e_i \doteq e_i,\quad \forall i \neq  1,\quad M_2e_1 \doteq \sigma e_1
\]
Here, $\sigma \in \{-1,1\}$ is suitably chosen to have the required positivity of the determinant. Finally, set $M \doteq M_2M_1$. Due to Lemma \ref{lem:LA}\eqref{DAorth}, we have that $D\A(D\tilde u) = MD\A(Du)$ pointwise a.e., and hence $\dv(D\A(D\tilde u)) = 0$. The above reasoning then yields the required regularity for $\tilde u$ and hence for $u$. This concludes the proof of this case, and thus of the present Proposition.
\end{proof}
\section{From outer to inner variations}\label{s:OUTTO}

In this section we show the factorization of $u$ as a harmonic map and a quasiconformal map in Lemma \ref{lem:qc} and the fact that the inverse of a solution to the outer variation system solves the inner variations system in Lemma \ref{lem:outtoin}.

\begin{lemma}\label{lem:qc}
Let $B_1 \subset \R^2$ and $u \in \Lip(B_1,\R^n)$. Suppose $u$ solves
\[
\dv(D\A(Du)) = 0 \text{ in the sense of distributions on $B_1$}.
\]
Then, there exists a quasiconformal map $\phi: B_1 \to \phi(B_1)$ such that
\[
v \doteq u\circ \phi^{-1}
\]
is harmonic in $U \doteq \phi(B_1)$.
\end{lemma}
\begin{proof}
Define the metric tensor $g \doteq \id + Du^TDu = \left(\begin{array}{cc} E& F\\ F & G\end{array}\right)$. The Lipschitz assumption on $u$ yields $\id \le g \le \Lambda \id$ at a.e. point of $B_1$ in the sense of quadratic forms. Using complex notation, we set
\[
\mu(z) \doteq \frac{E-G+2iF}{E+G+2{\sqrt {EG-F^{2}}}}.
\]
Introduce the \emph{Wirtinger derivatives} for $f:\Omega \subset \R^2\to\R^2$:
\[
f_z \doteq \frac{1}{2}[(\partial_1f_1 + \partial_2f_2) + i(\partial_1f_2 - \partial_2f_1)]\; \text{ and }\; f_{\bar z} \doteq \frac{1}{2}[(\partial_1f_1 - \partial_2f_2) + i(\partial_1f_2 + \partial_2f_1)].
\]
By direct computations, one verifies that
\[
|\mu|(z) \le k = k(\Lambda) < 1 \text{ for a.e. }z \in B_1.
\]
Therefore, after extending $\mu$ to $0$ outside $B_1$, we can invoke \cite[Theorem 5.3.2]{AIM} to find a unique quasiconformal map $\phi \in W^{1,2}_{\loc}(\R^2,\R^2)$ fulfilling a.e. in $\R^2$
\begin{equation}\label{phimu}
\phi_{\bar z} = \mu(z)\phi_z,
\end{equation}
and the normalization condition $\phi(z) = z + O\left(z^{-1}\right)$ for large $z$. Now \eqref{phimu} implies that $\det(D\phi)(z) \ge 0$ and hence $\det(D\phi)(z) > 0$ a.e. in $\R^2$, since the Jacobian of a quasiconformal map is non-zero a.e., \cite[Theorem 16.10.1]{IWAMA}. Set $U \doteq \phi(B_1)$ as in the statement of the ongoing lemma, and note that $U$ is open since $\phi$ is a homeomorphism. Passing to real notation, we find that \eqref{phimu} implies a.e. in $B_1$
\begin{equation}\label{gconf}
g = \rho D\phi^TD\phi, \quad \text{ with } \rho \doteq \frac{\sqrt{\det(g)}}{\det(D\phi)}.
\end{equation}
For details, see \cite[Theorem 10.1.1]{AIM}. By the first line of \eqref{sys0}, we know that criticality for outer variations corresponds to asking
\begin{equation}\label{ucrit}
\int_{B_1}\left\langle\sqrt{\det(g)}Dug^{-1},D\psi\right\rangle dx = 0,\quad \forall \psi \in C^\infty_c(B_1,\R^n).
\end{equation}
By approximation, the same holds if $\psi \in W^{1,2}_0(B_1,\R^n)$. We are now in position to show that $v \doteq u \circ \phi^{-1}$ is harmonic on $U$. In the subsequent computations, we will be using Theorem \ref{thm:quas} freely. Fix any $\eta \in C^\infty_c(U,\R^2)$. Then:
\begin{align*}
\int_{U}\langle Dv,D\eta\rangle dy &= \int_{U}\langle Du\circ \phi^{-1}(D\phi)^{-1}\circ \phi^{-1},D\eta\rangle dy \\
&= \int_{B_1}\det(D\phi)\langle Du(D\phi)^{-1},D\eta\circ \phi\rangle dx = \int_{B_1}\det(D\phi)\langle Du(D\phi)^{-1},D(\eta\circ \phi)(D\phi)^{-1}\rangle dx\\
& =\int_{B_1}\det(D\phi)\langle Du(D\phi)^{-1}(D\phi)^{-T},D(\eta\circ \phi)\rangle dx \overset{\eqref{gconf}}{=} \int_{B_1}\sqrt{\det(g)}\langle Dug^{-1},D(\eta\circ \phi)\rangle dx,
\end{align*}
where in the second line we used the change of variables $y = \phi(x)$ that is justified by, for instance, \cite[Lemma 2.4]{INV}. Since $\phi$ is a homeomorphism, $\phi(\partial B_1) \subset \partial U$. Thus, $\eta \circ \phi = 0$ on $\partial B_1$. Hence, $\eta \circ \phi \in W^{1,2}_0(B_1,\R^n)$ and \eqref{ucrit} concludes the proof.
\end{proof}

\begin{lemma}\label{lem:outtoin}
Let $B_1 \subset \R^2$ and $u \in \Lip(B_1,\R^n)$. Assume for $u$ a splitting as in \eqref{splitu}-\eqref{splitpsi}-\eqref{splitu2}, i.e.
\begin{equation}\label{splitulem}
u = \left(\begin{array}{cc} f \\ \psi\circ f\end{array}\right),
\end{equation}
for a quasiconformal $f: B_1 \to \R^2$ and a smooth $\psi: f(B_1) \to \R^{n-2}$. Assume in addition that $u$ solves
\begin{equation}\label{eq:out1}
\dv(D\A(Du)) = 0 \text{ in the sense of distributions on $B_1$}.
\end{equation}
Then, using notation \eqref{X|Y}, $w \doteq f^{-1}$ solves
\[
\dv(B(Dw|D\psi)) = 0 \text{ in the sense of distributions on $f(B_1)$}.
\]
\end{lemma}
\begin{proof}
Let $V \doteq f(B_1)$. We start by noticing that, since $w$ is quasiconformal, $B(Dw|D\psi) \in L^\infty(V)$ by Lemma \ref{lem:LA}\eqref{Bdd}. Notice that
\[
Dw(y) = (Df)^{-1}\circ w(y)
\]
by Theorem \ref{thm:quas}. By Lemma \ref{lem:LA}\eqref{xoxo}, for a.e. $y \in V$,
\begin{equation}\label{Bwu}
B(Dw|D\psi)(y) =  \frac{Df(w(y))B(Df(w(y))|D\psi(y) Df(w(y)))Df(w(y))^T}{\det(Df(w(y)))}.
\end{equation}
Therefore, given any $\eta \in C^\infty_c(V,\R^2)$, we have
\begin{equation}\label{Bwu2}
\begin{split}
\int_{V}\langle B(Dw|D\psi),D\eta\rangle dy &\overset{\eqref{Bwu}}{=} \int_{V}\left\langle\frac{Df(w(y))B(Df(w(y))|D\psi(y) Df(w(y)))Df(w(y))^T}{\det(Df(w(y)))}, D\eta\right\rangle dy \\
&= \int_{B_1}\langle Df(x)B(Df(x)|D\psi(f(x)) Df(x))Df(x)^T, D\eta \circ f\rangle dx \\
& = \int_{B_1}\langle DfB(Df|D(\psi\circ f))Df^T, D\eta \circ f\rangle dx\\
& \overset{\eqref{X|Y}-\eqref{splitulem}}{=}\int_{B_1}\langle DfB(Du)Df^T, D\eta \circ f\rangle dx\\
&=\int_{B_1}\langle DfB(Du), D(\eta \circ f)\rangle dx.
\end{split}
\end{equation}
We notice that $DfB(Du)$ is obtained considering the first two rows of $D\A(Du) =DuB(Du)$, compare \eqref{shortexpDA}, and thus $DfB(Du)$ is divergence-free since $D\A(Du)$ was. With a reasoning analogous to the one used in Lemma \ref{lem:qc}, we conclude that the last expression in \eqref{Bwu2} equals $0$, as wanted.
\end{proof}

\section{Regularity for solutions with small determinant}\label{s:SMALLDET}

This section is devoted to proving Proposition \ref{prop:regsmalldet}, which asserts the regularity of Lipschitz solutions with small determinant to the outer variations equations of the area. The idea behind it is that the area functional becomes uniformly convex close to matrices with small determinant, as can heuristically be seen by the fact that
\[
\A(X) = \sqrt{1 + |X|^2},\quad \text{ for all singular $X \in \R^{n\times 2}$}.
\]

\begin{rem}
A similar philosophy appears in \cite{WANG,WAN1}, where existence and regularity results for the minimal surface system are obtained in general dimension and codimension when the area-decreasing condition (products of pairs of distinct principal values of $Du$ are all at most one in absolute value) is satisfied.
\end{rem}

\begin{prop}\label{prop:regsmalldet}
Let $u \in \Lip(B_1,\R^n)$ be such that $\|Du\|_{\infty} \le \Lambda$ and
\begin{equation}\label{eq:out2}
\dv(D\A(Du)) = 0 \text{ in the sense of distributions on $B_1$}.
\end{equation}
There exists a constant $\eps_2 = \eps_2(\Lambda) > 0$ such that if $\|\det(Du^{ab})\|_\infty \le \eps_2$ for all $1\le a \le b \le n$, then $u \in C^{1,\alpha}(B_{1/2},\R^2)$.
\end{prop}
\begin{proof}
Let us start by remarking that, if $|X|,|Y| \le \Lambda$, $\rank(X) = \rank(Y) \le 1$ and $|Z| = 1$, then there exists $0 < \lambda = \lambda(\Lambda)$ such that
\begin{equation}\label{areaconv0}
\langle D\A(X) -D\A(Y),X-Y\rangle \ge 2\lambda|X-Y|^2
\end{equation}
and
\begin{equation}\label{areaconv0diff}
D^2\A(X)[Z,Z] = \lim_{t \to 0}\frac{\langle D\A(X + tZ) - D\A(X),Z\rangle}{t} \ge 2\lambda.
\end{equation}
Indeed, to show \eqref{areaconv0} we can exploit the explicit form of $D\A$ of \eqref{eq:expDA} to write
\[
D\A(X) = \frac{X}{\sqrt{1 + |X|^2}}, \quad D\A(Y) = \frac{Y}{\sqrt{1 + |Y|^2}},
\]
due to the fact that $\rank(X) = \rank(Y) \le 1$. Hence, \eqref{areaconv0} holds by the uniform convexity on compact sets of the function $f(X) = \sqrt{1 + |X|^2}$. To show \eqref{areaconv0diff}, define $g(t) \doteq \langle D\A(X + tZ),Z\rangle$. By \eqref{eq:expDA}:
\[
g(t) = \frac{\langle X + tZ,Z\rangle + \sum_{1\le a \le b \le n}\det(X^{ab} + t Z^{ab})\langle C_{ab}(X + tZ), Z\rangle}{\A(X + tZ)}.
\]
Therefore, using that $\rank(X) \le 1$, $|Z| = 1$ and \eqref{eq:expA}-\eqref{eq:expDA}, 
\begin{align*}
g'(0) &= \frac{1 + (\sum_{1\le a \le b \le n}\langle C_{ab}(X), Z\rangle)^2}{\A(X)} - \frac{\langle X,Z\rangle}{\A^2(X)}\langle D\A(X),Z\rangle \\
&= \frac{1 + (\sum_{1\le a \le b \le n}\langle C_{ab}(X), Z\rangle)^2}{\A(X)} - \frac{\langle X,Z\rangle^2}{\A^3(X)}\\
&\ge \frac{\A^2(X) - \langle X,Z\rangle^2}{\A^3(X)} \ge \frac{\A^2(X) -|X|^2}{\A^3(X)} = \frac{1}{\A^3(X)}.
\end{align*}
This shows \eqref{areaconv0diff}.
\\
\\
We exploit this observation to show that, given $\Lambda > 0$, there exists $0 < \eps_2 = \eps_2(\Lambda)$ such that if $|X|,|Y| \le \Lambda$ and $|\det(X^{ab})|,|\det(Y^{ab})| \le \eps_2$ for all $1 \leq a \leq b \leq n$, then
\begin{equation}\label{areaconv}
\langle D\A(X) -D\A(Y),X-Y\rangle \ge \lambda|X-Y|^2.
\end{equation}
If we manage to prove this inequality, then the proof of the interior regularity of $u$ is standard, see for instance the proof of \cite[Proposition 8.6]{GM} or the one of Theorem \ref{teo:regin} below. To show \eqref{areaconv}, we proceed by contradiction. If the statement of \eqref{areaconv} is false, then we find sequences $X_n,Y_n$ with $X_n\neq Y_n, |X_n|,|Y_n| \le \Lambda$ satisfying the following:
\begin{enumerate}[(i)]
\item $\det(X_n^{ab}),\det(Y_n^{ab}) \to 0, \quad \forall 1\le a \le b \le n$;\label{111}
\item $X_n \to X$, $Y_n \to Y$, $\frac{X_n - Y_n}{|X_n - Y_n|} \to Z$;\label{333}
\item $\langle D\A(X_n) -D\A(Y_n),X_n-Y_n\rangle \le \lambda|X_n-Y_n|^2$.\label{444}
\end{enumerate}
First, suppose $X \neq Y$. Then, in the limit \eqref{444} contradicts \eqref{areaconv0} due to \eqref{111}. We can then suppose $X = Y$. Dividing by $|X_n - Y_n|^2$ in \eqref{444} and sending that inequality to the limit, we find
\begin{equation}\label{contrafi}
D^2\A(X)[Z,Z] \le \lambda,
\end{equation}
Here $Z$ is the same matrix of $\eqref{333}$. Since $|Z| = 1$, $X$ is singular and $|X| \le \Lambda$, \eqref{contrafi} contradicts \eqref{areaconv0diff} and the proof is finished.\end{proof}

\begin{rem}
More generally, for any $m$ and  $n$ and any $R > 0$, there is a smooth, uniformly convex function on $\mathbb{R}^{n \times m}$ that agrees with $\A$ in a small neighborhood of the rank-one matrices of norm at most $R$. The key observation is that $\A$ touches the locally uniformly convex function $\sqrt{1 + |X|^2}$ from above on the rank-one matrices. 
As a consequence of this observation and general theory, Lipschitz weak solutions to the minimal surface system in a domain of dimension $m$ with gradient lying in a small neighborhood of the rank-one matrices are smooth away from a singular set with vanishing $(m-2)$-dimensional Hausdorff measure (see e.g. \cite[Theorem 9.1]{GM}). Arguing as in Lemma \ref{simpMT} we see that the solution is stationary, and since the assumption of closeness to rank-one matrices implies that the area-decreasing condition is satisfied, the solution is in fact smooth \cite[Theorem 4.1]{WANG}.
\end{rem}

\section{Regularity for solutions to inner variation equations}\label{s:SOLIN}

The main result of this section is Theorem \ref{teo:regin}, which shows regularity of quasiconformal solutions to the inner variation system. Its proof is based on the technical Proposition \ref{prop:regdiffinc}, and is close in spirit to \cite[Theorem 3]{SAP} and \cite[Theorem 6.5]{RT}.

\begin{prop}\label{prop:regdiffinc}
Let $X,Y \in \R^{2\times 2},M \in \R^{(n-2)\times 2}$. Assume the following properties for $K,\eps_3,L > 0$:
\begin{align}
&|X|^2 \le K\det(X), \;|Y|^2 \le K\det(Y), \label{XYqc}\\
&\det(X),\det(Y) \ge \eps_3>0,\label{eps3}\\
&|M| \le L.\label{BDD}
\end{align}
Then, there exists $C_1 > 0$ and $\delta > 0$ depending on $K, \eps_3, L$ such that
\begin{equation}\label{sptnull}
C_1\min\{|X|^2,|Y|^2\}|B(X|M)-B(Y|M)|^2 + \det(X-Y) \ge \delta|X-Y|^2.
\end{equation}
\end{prop}

Since the proof of this Proposition is rather lengthy, we postpone it to the end of the section.

\begin{Teo}\label{teo:regin}
Let $w \in \Lip(B_1,\R^2)$ be quasiconformal, $\det(Dw)\ge c > 0$ a.e. in $B_1$, $\psi \in C^\infty(B_1,\R^{n-2})$ and
\begin{equation}\label{eq:in2}
\dv(B(Dw|D\psi)) = 0 \text{ in the sense of distributions on $B_1$},
\end{equation}
then $w \in C^{1,\alpha}(B_{1/2},\R^2)$.
\end{Teo}
\begin{proof}
 First we need to recall some properties of the matrix-valued map $x\mapsto B(Dw|D\psi)(x)$, see \cite[Section 3]{RT} for further details. By Lemma \ref{lem:LA}\eqref{signs}-\eqref{Bdd}, \eqref{eq:in2} implies that
\[
B(Dw|D\psi) = \cof(D^2z),
\]
where $z: B_1 \to \R$ is a convex, $W^{2,\infty}$ solution to the Monge-Amp\`ere equation
\[
\det(D^2z) = 1.
\]
Therefore, $z$ is smooth (see e.g. \cite[Sections 2-3]{FIG} and the references therein), and hence so is $x\mapsto B(Dw|D\psi)(x)$. Since $\det(Dw) \ge c$ a.e., $w$ is quasiconformal and $D\psi$ is smooth, hence bounded, we see that for a.e. $x,y \in B_1$, $X = Dw(x)$, $Y = Dw(y)$, $M = D\psi(x)$ fulfill the assumptions of Proposition \ref{prop:regdiffinc}. This allows us to start the proof of the $C^{1,\alpha}$ regularity of $w$ in $B_{1/2}$.
\\
\\
We consider any $h \in \R^2$ such that $0 < |h| < \frac{1}{4}$ and denote, for any function $g$, $g^h(x) \doteq \frac{g(x+ h) - g(x)}{|h|}$. We take any $\eta \in C^\infty_c(B_{3/4})$, $\eta(x) \ge 0\; \forall x \in B_1$, and finally set $m_h(x) \doteq \min\{|Dw|(x),|Dw|(x + h)\}$. We start by noticing that
\begin{equation}\label{nl-1}
\begin{split}
\int_{B_1}\eta^2m_h(x)^2|B(Dw(x + h)|&D\psi(x)) - B(Dw(x)|D\psi(x))|^2dx \\
&\le 2\int_{B_1}\eta^2m_h(x)^2|B(Dw(x + h)|D\psi(x + h)) - B(Dw(x)|D\psi(x))|^2dx  \\
&\qquad+ 2 \int_{B_1}\eta^2m_h(x)^2|B(Dw(x + h)|D\psi(x + h)) - B(Dw(x + h)|D\psi(x))|^2dx\\
&\le C|h|^2\int_{B_1}\eta^2m_h(x)^2dx \\
&\qquad+ C\int_{B_1}\eta^2m_h(x)^2|B(Dw(x + h)|D\psi(x + h)) - B(Dw(x+h)|D\psi(x))|^2dx\\
&\le C|h|^2\int_{B_1}\eta^2m_h(x)^2dx + C\int_{B_1}\eta^2\frac{m_h(x)^2}{1 + |Dw|(x + h)}|D\psi(x + h) - D\psi(x)|^2dx\\
&\le C|h|^2\int_{B_1}\eta^2m_h(x)^2dx  \le C|h|^2 \int_{B_1}|Dw|^2(x)dx.
\end{split}
\end{equation}
In this chain of inequalities and in the rest of the proof, $C>0$ is a constant that may vary line-by-line. In passing from the first to the second inequality we used the smoothness of $x \mapsto B(Dw|D\psi)(x)$, while in passing from the second to the third we employed Lemma \ref{lem:LA}\eqref{gradBdd}. We can finally start the main part of the proof.
\\
\\
Let $b \in \R^2$ be any constant. Since the determinant is a Null Lagrangian, see \cite[Theorem 2.3(ii)]{DMU}, then:
\begin{equation}\label{nl0}
\begin{split}
0 &=\int_{B_1}\det(D (\eta (w^h - b)))(x)dx \\
&= \int_{B_1}\eta^2\det(D w^h(x))dx + \int_{B_1} \langle\cof^T((w^h - b)\otimes D\eta), \eta Dw^h\rangle dx.
\end{split}
\end{equation}
Hence
\begin{equation}\label{nulllag}
\begin{split}
\int_{B_1} |w^h - b||D\eta||\eta|| Dw^h|dx &\overset{\eqref{nl0}}{\ge} \int_{B_1}\eta^2\det(D w^h(x))dx \\
&\overset{\eqref{sptnull}}{\ge} \delta\int_{B_1}\eta^2|Dw^h|^2dx \\
&\qquad- \frac{C}{|h|^2}\int_{B_1}\eta^2m_h(x)^2|B(Dw(x + h)|D\psi(x)) - B(Dw(x)|D\psi(x))|^2dx\\
&\overset{\eqref{nl-1}}{\ge} \delta\int_{B_1}\eta^2|Dw^h|^2dx - C\int_{B_1}\eta^2|Dw|^2dx.
\end{split}
\end{equation}
Inequality \eqref{nulllag} implies, after using Young's inequality:
\begin{equation}\label{basicin1}
\int_{B_1}\eta^2|Dw^h|^2dx \le C\int_{B_1}\eta^2|Dw|^2dx + C\int_{B_1}|w^h - b|^2|D\eta|^2dx.
\end{equation}
Choose first $b = 0$. Then, we can invoke \cite[Proposition 9.3]{BRE} to infer $u \in W_{\loc}^{2,2}(B_{3/4})$. We will now show $u \in W^{2,2 +\varepsilon}_{\loc}(B_{3/4})$ for some $\eps > 0$. 
Take any square $Q$ of side $r$ such that $2Q \subset B_{3/4}$. Here, $\lambda Q$ is the square centered at the center of $Q$ and with side $\lambda r$, for any $\lambda > 0$. Suppose $\eta \in C_c^\infty(2Q)$, $\eta \equiv 1$ on $Q$, and $|D\eta(x)| \le Cr^{-1}, \forall x \in 2Q$. Combining this choice with the Poincarè-Sobolev inequality on cubes, we obtain from \eqref{basicin1} and a suitable choice of $b$:
\[
\int_{Q}|Dw^h|^2dx \le C\int_{2Q}|Dw|^2dx + \frac{C}{r^2}\left(\int_{2Q}|Dw^h|dx\right)^2.
\]
Letting $|h| \to 0$, we finally get
\begin{equation*}\label{basicin}
\int_{Q}|D^2w|^2dx \le C\int_{2Q}|Dw|^2dx + \frac{C}{r^2}\left(\int_{2Q}|D^2w|dx\right)^2.
\end{equation*}
Now Gehring's Lemma as stated for instance in \cite[Proposition 5.1]{MGR} concludes the proof of the higher integrability of $D^2w$. Thus, $Dw \in W_{\loc}^{1,2+\eps}(B_{3/4},\R^{2\times 2})$ and hence belongs to $C^\alpha(B_{1/2})$, as wanted.
\end{proof}

\begin{Rem}
Suppose that $n =2$, so that \eqref{eq:in2} reduces to the more classical inner variations equations $\dv(B(Dw)) = 0.$ Theorem \ref{teo:regin} states that quasiconformal solutions to the inner variations equations is regular\footnote{In the statement of Theorem \ref{teo:regin}, we have the additional assumption $\det(Dw) \ge c > 0$, which is a natural feature of our problem, see \eqref{detDwbdd}. However with a slightly different proof one can show that, at least for the case $n= 2$, this assumption is not needed, and hence one can compare more precisely our result with the one of \cite{HEL}.}. This result can be compared with the analogous results for the Dirichlet energy of F. H\'elein \cite{HEL} and T. Iwaniec, L. V. Kovalev, J. Onninen \cite[Theorem 1.3]{IKO}. The first cited result is the most directly related, since there the map is assumed to be quasiconformal and critical with respect to inner variations of the Dirichlet energy, while the second shows the same assuming the weaker condition that the map is continuous, open and discrete. It is also interesting to notice that our proof is substantially different from both the one of \cite{HEL} and the one of \cite{IKO}.
\end{Rem}

We conclude this section by proving Proposition \ref{prop:regdiffinc}.

\begin{proof}[Proof of Proposition \ref{prop:regdiffinc}]
Let us divide the proof into steps. The expression $A \lesssim B$ means $A \le CB$, for some $C = C(K,\eps_3,L) > 0$. The notation $A \sim B$ means $A \lesssim B$ and $B \lesssim A$. When it appears explicitly, the constant $C$ will always depend solely on $K, \eps_3$ and $L$ and can change line-by-line.
\\
\\
\fbox{Step 1: Reduction to $M = 0$.} We start by remarking that it is enough to show \eqref{sptnull} in the case $M = 0$. Indeed, assume for a moment that we know the statement of the Proposition for every $\eps_3,K,X,Y$ and $M = 0$. Then, consider
\[
S = \sqrt{\id + M^TM}.
\]
In the sense of quadratic forms
\begin{equation}\label{boundS}
\id \le S \le C\id.
\end{equation}
We employ Lemma \ref{lem:LA}\eqref{BsimpS} to write
\begin{equation}\label{BY0}
|B(X|M) - B(Y|M)|^2 = \det(S)^2|S^{-1}(B(XS^{-1}|0) - B(YS^{-1}|0))S^{-1}|^2\overset{\eqref{boundS}}{\sim} |B(XS^{-1}|0) - B(YS^{-1}|0)|^2.
\end{equation}
Notice now that $XS^{-1},YS^{-1}$ fulfill \eqref{XYqc}-\eqref{eps3} with new constants $K'$ and $\eps_4$ only depending on the previous constants $K,\eps_3,L$. Using our assumption on the fact that the present proposition holds in the case $M = 0$, we can estimate
\[
C_1\min\{|XS^{-1}|^2,|YS^{-1}|^2\}|B(S^{-1}X|0)-B(YS^{-1}|0)|^2 + \det(XS^{-1}-YS^{-1}) \ge \delta|XS^{-1}-YS^{-1}|^2
\]
Using \eqref{boundS} and \eqref{BY0}, the general estimate \eqref{sptnull} follows. In the rest of the proof, we will assume $M=0$.
\\
\\
\fbox{Step 2: Simplifications in the case $M = 0$.} It suffices to show that, for all $X,Y$ fulfilling \eqref{XYqc}-\eqref{eps3}, there exists $\eps = \eps(K,\eps_3)>0$ such that, if
\begin{equation}\label{eq:assu}
|B(X|0)-B(Y|0)|^2\min\{|X|,|Y|\}^2\le \eps|X-Y|^2,
\end{equation}
then we can find $0<\delta <1$ as in the statement of this proposition such that
\begin{equation}\label{eq:detXY}
\det(X-Y)\ge \delta |X-Y|^2.
\end{equation}
If we show this statement, then \eqref{sptnull} follows choosing $C_1 \doteq 2/\eps$. Indeed, either \eqref{eq:assu} holds, and hence \eqref{eq:detXY} directly implies \eqref{sptnull}, or \eqref{eq:assu} does not hold, in which case
\[
C_1|B(X|0)-B(Y|0)|^2\min\{|X|,|Y|\}^2 + \det(X-Y) \ge C_1\eps|X-Y|^2 -|X-Y|^2 = |X-Y|^2\ge \delta|X-Y|^2.
\]
We can make a few simplifications before showing how \eqref{eq:assu} implies \eqref{eq:detXY}. 
\\
\\
First, we can always assume
\begin{equation}\label{YlessX}
|Y| \le |X|.
\end{equation}
Furthermore, we can always assume that there exists $c_0 = c_0(K) >0$ such that
\begin{equation}\label{c0XY}
c_0|X| \le |Y|.
\end{equation}
Indeed, there exists $c_0= c_0(K)$, $0<c_0<1$, such that for any couple of matrices $X,Y$ fulfilling \eqref{XYqc}-\eqref{YlessX}, the inequality
\begin{equation}\label{YcontroX}
|Y| \le c_0|X|,
\end{equation}
implies \eqref{eq:detXY} for any $\delta \le \frac{1}{2K}$. Indeed, suppose this is false, so that we can find sequences $X_n$ and $Y_n$ with $|X_n| \neq 0, \forall n$,
\begin{equation}\label{assu22}
|Y_n| \le \frac{1}{n}|X_n|,\quad \text{ $X_n,Y_n$ fulfill \eqref{XYqc},\quad and $\frac{X_n}{|X_n|} \to X$},
\end{equation}
but
\begin{equation}\label{contra2}
\det(X_n-Y_n) \le \frac{1}{2K}|X_n-Y_n|^2.
\end{equation}
Then, clearly $|X| = 1$ and $X$ fulfills \eqref{XYqc} as well. Dividing by $|X_n|^2$ in \eqref{contra2} and passing to the limit as $n\to \infty$, we can exploit \eqref{assu22} to get a contradiction with \eqref{XYqc}.
\\
\\
The second simplification we can perform is to exploit the principal value decomposition and Lemma \ref{lem:LA}\eqref{orthr}-\eqref{orth} to reduce us to study only the case in which $X$ is diagonal with positive entries:
\begin{equation}\label{X:diag}
X = \left(\begin{array}{cc} a & 0\\0&b\end{array}\right), \quad a,\, b > 0.
\end{equation}
We can start with the main part of the proof. Recall that we assume \eqref{XYqc}-\eqref{eps3}-\eqref{YlessX}-\eqref{c0XY}, and $X$ has the special form $\eqref{X:diag}$. Under this set of assumptions, we need to show that \eqref{eq:assu} implies \eqref{eq:detXY}. We will first deal with the case in which $B(Y|0)$ is as well diagonal, which in view of \eqref{eq:Bexp} is equivalent to assuming that the columns of $Y$ are orthogonal. We will deal with the general case afterwards.
\\
\\
\fbox{Step 3: $B(Y|0)$ diagonal.} Write $Y = (Y^1|Y^2)$. As $Y^1$ and $Y^2$ are orthogonal, we can find $v \in \mathbb{S}^1$ and $\alpha,\beta$ such that
\begin{equation}\label{Yorto}
Y = (\alpha v | \beta Jv).
\end{equation}
Here, we defined the symplectic matrix
\[
J\doteq
\left(
\begin{array}{cc}
0 & -1\\
1 & 0
\end{array}
\right).
\]
Up to choosing $-v$, we can assume $\alpha, \beta > 0$. Notice that \eqref{XYqc}-\eqref{YlessX}-\eqref{c0XY} imply
\begin{equation}\label{allcomp}
|X|\sim |Y| \sim \alpha \sim \beta \sim a \sim b.
\end{equation}
If $B$ is diagonal, \eqref{eq:assu} carries only one non-trivial information:
\begin{equation}\label{alphabeta}
\left|\frac{\sqrt{1 + \beta^2 }}{\sqrt{1 + \alpha ^2}} - \frac{\sqrt{1 + b^2}}{\sqrt{1 + a ^2}}\right|^2|Y|^2 \le \eps |X-Y|^2,
\end{equation}
as can be seen exploiting \eqref{eq:Bexp}. We wish to show that \eqref{alphabeta} yields
\begin{equation}\label{ts1}
|Y|^{-2}\left|\frac{1 + a^2}{1 +b^2}(\beta^2 - b^2) - (\alpha^2 - a^2)\right|^2 \lesssim \eps |X-Y|^2.
\end{equation}
We have:
\begin{equation}\label{ts2}
\begin{split}
\left|\frac{\sqrt{1 +\beta^2}}{\sqrt{1 + \alpha ^2}} - \frac{\sqrt{1 + b^2}}{\sqrt{1 + a ^2}}\right||Y| &= \left|\frac{\sqrt{1 +\beta^2}\sqrt{1 + a^2} - \sqrt{1 + b^2}\sqrt{1 + \alpha^2}}{\sqrt{1 + \alpha ^2}\sqrt{1 + a^2}}\right||Y|\\
&\overset{\eqref{eps3}-\eqref{allcomp}}{\ge} C|Y|^{-1} \left|\sqrt{1 +\beta^2}\sqrt{1 + a^2} - \sqrt{1 + b^2}\sqrt{1 + \alpha^2}\right|\\
& \ge C|Y|^{-3}\left|(1 +\beta^2)(1 + a^2) - (1 + b^2)(1 + \alpha^2)\right|\\
& = C|Y|^{-3}(1 + b^2)\left|\frac{1 + a^2}{1 + b^2}(\beta^2 - b^2) - (\alpha^2 - a^2)\right| \\
&\overset{\eqref{allcomp}}{\ge} C|Y|^{-1}\left|\frac{1 + a^2}{1 + b^2}(\beta^2 - b^2) - (\alpha^2 - a^2)\right|
\end{split}
\end{equation}
\eqref{alphabeta}-\eqref{ts2} imply \eqref{ts1}. We are finally in position to prove \eqref{eq:detXY}. Let $v_1$ be the first component of $v$ and write
\begin{equation}\label{XYnorm}
\begin{split}
|X-Y|^2 &= \alpha^2 + a^2 -2v_1a\alpha + \beta^2 + b^2 - 2v_1b\beta = (\alpha - a)^2 + (\beta - b)^2 + 2(1-v_1)(a\alpha + b\beta)\\
& \ge (a-\alpha)^2 + (b-\beta)^2,
\end{split}
\end{equation}
where in the last inequality we used the fact that $v_1 \le 1$ and $a\alpha,b\beta > 0$. Moreover,
\begin{equation}\label{XYdetdet}
\begin{split}
\det(X-Y) &= (ae_1 - \alpha v, be_1 - \beta v) = ab + \alpha \beta - v_1(\alpha b +a\beta) = (a - \alpha)(b-\beta) +(1-v_1)(\alpha b +a\beta)
\end{split}
\end{equation}
By \eqref{allcomp}, we have
\begin{equation}\label{easyterm}
(a\alpha + b\beta) \lesssim (\alpha b +a\beta).
\end{equation}
Furthermore, first multiplying and diving by $a + \alpha \neq 0$ by \eqref{eps3} and then summing and subtracting $\frac{1 + a^2}{1 + b^2}(\beta^2-b^2)$:
\begin{equation}\label{diffterm0}
\begin{split}
(a-\alpha)(b-\beta) = \frac{b-\beta}{a+\alpha}(a^2-\alpha^2) =  \frac{\beta-b}{a+\alpha}\frac{1 + a^2}{1 + b^2}(\beta^2 - b^2) +  \frac{b-\beta}{a+\alpha}\left(a^2-\alpha^2 + \frac{1 + a^2}{1 + b^2}(\beta^2 - b^2)\right).
\end{split}
\end{equation}
By \eqref{allcomp} we see that
\begin{equation}\label{firstaddiff}
\frac{\beta-b}{a+\alpha}\frac{1 + a^2}{1 + b^2}(\beta^2 - b^2) = (b-\beta)^2\frac{b+ \beta}{a+\alpha}\frac{1 + a^2}{1 + b^2} \ge C(b-\beta)^2.
\end{equation}
The second addendum of \eqref{diffterm0} is estimated as:
\begin{equation}\label{secondaddiff}
\left|\frac{b-\beta}{a+\alpha}\left(a^2-\alpha^2 + \frac{1 + a^2}{1 + b^2}(\beta^2 - b^2)\right)\right| \overset{\eqref{eps3}-\eqref{allcomp}-\eqref{ts1}-\eqref{XYnorm}}{\lesssim} \sqrt{\eps}|X-Y|^2
\end{equation}
A symmetric reasoning yields as well
\begin{equation}\label{a-alpha}
(a-\alpha)(b-\beta) \ge C(a-\alpha)^2 - C\sqrt{\eps}|X-Y|^2.
\end{equation}
Combining \eqref{easyterm}-\eqref{diffterm0}-\eqref{firstaddiff}-\eqref{secondaddiff}, we obtain
\[
|X-Y|^2 \le C\det(X-Y) + C\sqrt{\eps}|X-Y|^2.
\]
By choosing $\eps$ sufficiently small we finally get \eqref{eq:detXY} if the columns of $Y$ are orthogonal. We now move to the general case.
\\
\\
\fbox{Step 4: The general case.} Let $Y =(Y^1|Y^2) \in \R^{2\times 2}$, and consider the following splitting
\[
Y = Y_{e} + Y_o \doteq \left(0\left|\frac{(Y^1,Y^2)}{|Y^1|^2}Y^1\right.\right) + \left(Y^1\left|\frac{(JY^1,Y^2)}{|Y^1|^2}JY^1\right.\right) = \left(0\left|\frac{(Y^1,Y^2)}{|Y^1|^2}Y^1\right.\right) + \left(Y^1\left|\frac{\det(Y)}{|Y^1|^2}JY^1\right.\right).
\]
Considering the off-diagonal elements of $B(Y|0)$ and using the explicit expression \eqref{eq:Bexp} and \eqref{eq:assu}, we obtain:
\begin{equation}\label{scalprod}
\left|\frac{(Y^1,Y^2)}{\A(Y|0)}\right|^2|Y|^2 \le \eps|X-Y|^2.
\end{equation}
Since $Y$ fulfills \eqref{XYqc}-\eqref{eps3}, we see that
\begin{equation}\label{Aquad}
\A(Y|0) \sim |Y|^2.
\end{equation}
Thus, \eqref{scalprod} yields
\begin{equation}\label{scalprod1}
\left|(Y^1,Y^2)\right|^2|Y|^{-2} \lesssim \eps|X-Y|^2.
\end{equation}
Again since $Y$ fulfills $\eqref{XYqc}$, 
\begin{equation}\label{Yallcomp}
|Y^1| \sim |Y^2|\sim |Y|.
\end{equation}
Therefore, \eqref{scalprod1} yields
\begin{equation}\label{Yerr}
|Y_e|^2 \lesssim \eps|X-Y|^2.
\end{equation}
In turn, if $\eps$ is sufficiently small, \eqref{Yerr} implies
\begin{equation}\label{Yerro}
|Y_e|^2 \lesssim \eps|X-Y_o|^2.
\end{equation}
In particular, \eqref{eq:assu} and \eqref{Yerr} imply
\[
|B(X|0) - B(Y|0)|^2|Y|^2 \lesssim \eps|X-Y_o|^2
\]
Since $|Y_o|^2 \leq |Y|^2$ we conclude that
\begin{equation}\label{starting0}
|B(X|0) - B(Y|0)|^2|Y_o|^2 \lesssim \eps|X-Y_o|^2.
\end{equation}
Consider the segment
\[
\sigma(t) \doteq Y_o + tY_e,\quad \text{for }t\in[0,1].
\]
Then $\det(\sigma(t)) = \det(Y)$ for all $t$, as can be verified by direct computation. Furthermore, since $\langle Y_e, Y_o\rangle = 0$,
\begin{equation}\label{boundnormsig}
|Y_o| \le |\sigma(t)| \overset{\eqref{Yallcomp}}{\lesssim}|Y|,\quad \forall t \in [0,1].
\end{equation}
Therefore, $\sigma(t)$ fulfills
\begin{equation}\label{sigmat}
|\sigma(t)|^2 \lesssim |Y|^2 \overset{\eqref{XYqc}}{\lesssim}\det(Y) = \det(\sigma(t)), \quad \forall t \in[0,1].
\end{equation}
Now Lemma \ref{lem:LA}\eqref{gradBdd} implies
\begin{equation}\label{starting}
\begin{split}
|B(Y_o|0) - B(Y|0)|^2|Y_o|^2 = \left|\int_{0}^1DB(\sigma(t)|0)[\sigma'(t)]\right|^2dt|Y_o|^2 \lesssim  \left(\int_{0}^1\frac{|\sigma'(t)|}{1 + |\sigma(t)|}dt\right)^2|Y_o|^2 \\
\overset{\eqref{boundnormsig}-\eqref{sigmat}}{\lesssim} |Y_e|^2  \overset{\eqref{Yerro}}{\lesssim} \eps|X-Y_o|^2.
\end{split}
\end{equation}
Inequalities \eqref{starting}-\eqref{starting0} and the triangle inequality finally give us
\begin{equation}\label{starting1}
|B(X|0) - B(Y_o|0)|^2|Y_o|^2 \lesssim \eps|X-Y_o|^2.
\end{equation}
Using that $\det(Y_o) = \det(Y)$ and $|Y_o|^2 \leq |Y|^2$, we have that $Y_o$ fulfills \eqref{XYqc}-\eqref{eps3}. Thus, having chosen $\eps$ small, we can employ the previous step to deduce the existence of $\delta > 0$ such that
\begin{equation}\label{X-Y0}
\delta|X-Y_o|^2 \le \det(X-Y_o).
\end{equation}
Since $Y = Y_o + Y_e$ and $Y_e$ fulfills \eqref{Yerr}-\eqref{Yerro}, we can select $\eps$ even smaller and depending on $\delta$ to obtain \eqref{eq:detXY} from \eqref{X-Y0}. This concludes the proof of the general case, and hence of this Proposition.
\end{proof}

\appendix

\section{Proof of Lemma \ref{lem:LA}}\label{sec:PLA}

\begin{proof}
\eqref{orthr}-\eqref{orth} are immediate from \eqref{initial} and the same symmetries of the area integrand, and \eqref{DAorth} follows from these and \eqref{shortexpDA}. Also the properties in \eqref{signs} are clear from \eqref{initial}-\eqref{Basg}.
\\
\\
We now move to the proof of \eqref{xoxo}. We use \eqref{Basg} to write, for any $W \in \R^{2\times 2}$:
\begin{equation}\label{Binverse}
B^{-1}(W|Y) = \frac{\id + W^TW  + Y^TY}{\sqrt{\det(\id + W^TW + Y^TY)}}.
\end{equation}
Since
\[
\id + X^{-T}X^{-1}  + Y^TY =  X^{-T}(X^TX + \id + (YX)^TYX)X^{-1},
\]
we readily infer
\[
B^{-1}(X^{-1}|Y) = \det(X)X^{-T}B^{-1}(X|YX)X^{-1}.
\]
Taking the inverse on the left and on the right, we get
\[
B(X^{-1}|Y) = \det(X)^{-1}XB(X|YX)X^{T},
\]
as wanted.
\\
\\
To show \eqref{BsimpS}, we consider again \eqref{Binverse}. Since
\[
\id + X^TX + Y^TY = S^2 + X^TX = S(\id + (XS^{-1})^TXS^{-1})S,
\]
the result readily follows.
\\
\\
Finally, to show \eqref{Bdd}-\eqref{gradBdd}, we first notice that for $X,Y$ fulfilling 
\begin{equation}\label{XM}
|X|^2 \le K\det(X) \text{ and } |Y| \le L,
\end{equation}
we have, exploiting for instance the third form of $\A$ in \eqref{eq:expA},
\begin{equation}\label{areaboundqc}
c(1 + |X|^2)\le \A(X|Y) \le C(1 + |X|^2),
\end{equation}
for positive constants $c,C > 0$. In the following, $c$ and $C$ may vary line-to-line, but they are uniform constants depending solely on $n$, $L$ and $K$. Consider any polynomial of order at most two, $Q = Q(X,Y) : \R^{2\times 2}\times \R^{(n-2)\times 2} \to \R$ and set
\[
h(X,Y) \doteq \frac{Q(X,Y)}{\A(X|Y)}
\]
From the explicit expression \eqref{eq:Bexp} we see that it suffices to show that for all $X,Y$ fulfilling \eqref{XM}, $h$ can be estimated as
\[
|h|(X|Y) \le C \quad \text{ and }\quad |Dh|(X|Y) \le \frac{C}{1 + |X|}.
\]
However, the first inequality is obvious from \eqref{areaboundqc} and the fact that $Q$ is a polynomial of at most quadratic growth. Next, considering any derivative in a coordinate direction, say $\partial_i$, we have
\[
\partial_ih(X,Y) = \frac{\partial_iQ(X,Y)}{\A(X|Y)} - \frac{Q(X,Y)}{\A(X|Y)^2}\partial_i\A(X|Y).
\]
The required estimate then follows from the fact that $Q$ is at most quadratic, \eqref{XM}, and the estimate
\[
|\partial_i\A(X|Y)| \le C(1 + |X| + |Y|) \le C(1 + |X|),
\]
which can be found in \cite[Lemma 2.1(1)]{RT} or can be seen directly from \eqref{eq:expDA}.
\end{proof}

\bibliographystyle{plain}
\bibliography{InnerCrit}

\begin{bibdiv}
\begin{biblist}

\bib{AIM}{book}{
      author={Astala, K.},
      author={Iwaniec, T.},
      author={Martin, G.},
       title={{Elliptic Partial Differential Equations and Quasiconformal
  Mappings in the Plane (PMS-48)}},
      series={Princeton Mathematical Series},
   publisher={Princeton University Press},
        date={2008},
        ISBN={9781400830114},
         url={https://books.google.de/books?id=O\_3-otfrBW8C},
}

\bib{BAR}{article}{
      author={Barbosa, Joao Lucas~Marqu{\^{e}}s},
       title={{An extrinsic rigidity theorem for minimal immersions from $S^2$
  into $S^n$}},
        date={1979},
     journal={Journal of Differential Geometry},
      volume={14},
      number={3},
       pages={355\ndash 368},
}

\bib{BI}{article}{
      author={Bojarski, B.},
      author={Iwaniec, T.},
       title={{Analytical foundations of the theory of quasiconformal mappings
  in $\R^n$}},
        date={1983},
     journal={Annales Academiae Scientiarum Fennicae Series A I Mathematica},
      volume={8},
       pages={257\ndash 324},
}

\bib{BRE}{book}{
      author={Brezis, Haim},
       title={{F}unctional {A}nalysis, {S}obolev {S}paces and {P}artial
  {D}ifferential {E}quations},
   publisher={Springer},
        date={2010},
}

\bib{DEG}{article}{
      author={De~Giorgi, Ennio},
       title={Sulla differenziabilit{\`a} e l'analiticit{\`a} delle estremali
  degli integrali multipli regolari},
        date={1957},
     journal={Mem. Accad. Sci. Torino. Cl. Sci. Fis. Mat. Nat.},
}

\bib{DLDPKT}{article}{
      author={De~Lellis, Camillo},
      author={De~Philippis, Guido},
      author={Kirchheim, Bernd},
      author={Tione, Riccardo},
       title={Geometric measure theory and differential inclusions},
        date={2021-12},
     journal={Annales de la Facult{\'{e}} des sciences de Toulouse :
  Math{\'{e}}matiques},
      volume={30},
      number={4},
       pages={899\ndash 960},
}

\bib{DIMLER}{article}{
      author={Dimler, Bryan},
       title={{Partial regularity for Lipschitz solutions to the minimal
  surface system}},
        date={2023},
      eprint={2306.12586},
}

\bib{DOUG}{article}{
      author={Douglas, Jesse},
       title={{Solution of the problem of Plateau}},
        date={1931},
     journal={Transactions of the American Mathematical Society},
      volume={33},
      number={1},
       pages={263\ndash 321},
}

\bib{DUSE}{article}{
      author={Duse, Erik},
       title={{Generic Ill-posedness of the Energy-Momentum Equations and
  Differential Inclusions}},
        date={2023},
      eprint={2210.16533},
}

\bib{FIG}{misc}{
      author={Figalli, Alessio},
       title={The {M}onge-{A}mpère {E}quation and {I}ts {A}pplications},
         how={{Z}urich {L}ectures in {A}dvanced {M}athematics},
        date={2017},
}

\bib{DFC}{article}{
      author={Fischer-Colbrie, D.},
       title={Some rigidity theorems for minimal submanifolds of the sphere},
        date={1980},
        ISSN={0001-5962},
     journal={Acta Mathematica},
      volume={145},
       pages={29\ndash 46},
}

\bib{INV}{article}{
      author={Fonseca, I.},
      author={Gangbo, W.},
       title={{Local Invertibility of Sobolev Functions}},
        date={1995},
     journal={SIAM Journal on Mathematical Analysis},
      volume={26},
      number={2},
       pages={280\ndash 304},
      eprint={https://doi.org/10.1137/S0036141093257416},
         url={https://doi.org/10.1137/S0036141093257416},
}

\bib{GM}{misc}{
      author={Giaquinta, Mariano},
      author={Martinazzi, Luca},
       title={{An Introduction to the Regularity Theory for Elliptic Systems,
  Harmonic Maps and Minimal Graphs}},
        date={2012},
}

\bib{HEL}{article}{
      author={H{\'{e}}lein, Fr{\'{e}}d{\'{e}}ric},
       title={{Homéomorphismes quasi-conformes entre surfaces riemanniennes.
  (Quasi- conformal homeomorphisms between Riemannian surfaces)}},
        date={1988},
     journal={Comptes Rendus de l’Académie des Sciences. Série I},
}

\bib{HELBO}{book}{
      author={H{\'{e}}lein, Fr{\'{e}}d{\'{e}}ric},
       title={{Harmonic Maps, Conservation Laws and Moving Frames}},
   publisher={Cambridge University Press},
        date={2002},
}

\bib{HRT}{article}{
      author={Hirsch, Jonas},
      author={Tione, Riccardo},
       title={On the constancy theorem for anisotropic energies through
  differential inclusions},
        date={2021-04},
     journal={Calculus of Variations and Partial Differential Equations},
      volume={60},
      number={3},
}

\bib{IKO}{article}{
      author={Iwaniec, T.},
      author={Kovalev, L.~V.},
      author={Onninen, J.},
       title={{Hopf Differentials and Smoothing Sobolev Homeomorphisms}},
        date={2011-07},
     journal={International Mathematics Research Notices},
}

\bib{IWAMA}{book}{
      author={Iwaniec, Tadeusz},
      author={Martin, Gaven},
       title={Geometric function theory and non-linear analysis},
   publisher={Clarendon press},
        date={2001},
}

\bib{LOA}{article}{
      author={{L}awson, {H}.~{B}.},
      author={{O}sserman, {R}.},
       title={Non-existence, non-uniqueness and irregularity of solutions to
  the minimal surface system},
        date={1977},
     journal={Acta Math.},
      volume={139},
       pages={1\ndash 17},
}

\bib{MGR}{article}{
      author={{M}ariano {G}iaquinta},
      author={{G}iuseppe Modica},
       title={Regularity {R}esults for {S}ome {C}lasses of {H}igher {O}rder non
  {Linear} {E}lliptic {S}ystems},
        date={1979},
     journal={Journal für die reine und angewandte Mathematik},
       pages={145\ndash 149},
}

\bib{MOS}{article}{
      author={Moser, Jürgen},
       title={{A new proof of de Giorgi{\textquotesingle}s theorem concerning
  the regularity problem for elliptic differential equations}},
        date={1960-08},
     journal={Communications on Pure and Applied Mathematics},
      volume={13},
      number={3},
       pages={457\ndash 468},
}

\bib{DMU}{misc}{
      author={M{\"u}ller, Stefan},
       title={Variational {M}odels for {M}icrostructure and {P}hase
  {T}ransitions},
         how={Lectures at the CIME {S}ummer {S}chool {C}alculus of {V}ariations
  and {G}eometric {E}volution {P}roblems},
        date={1999},
}

\bib{SMVS}{article}{
      author={M{\"u}ller, Stefan},
      author={{\v S}ver{\'a}k, Vladimir},
       title={{C}onvex integration for {L}ipschitz mappings and counterexamples
  to regularity},
        date={2003},
        ISSN={0003-486X},
     journal={Annals of Mathematics},
      volume={157},
      number={3},
       pages={715\ndash 742},
      review={\MR{1983780}},
}

\bib{NASH}{article}{
      author={Nash, J.},
       title={{Continuity of Solutions of Parabolic and Elliptic Equations}},
        date={1958-10},
     journal={American Journal of Mathematics},
      volume={80},
      number={4},
       pages={931},
}

\bib{NASHPRE}{article}{
      author={Nash, John},
       title={{Parabolic Equations}},
        date={1957},
        ISSN={00278424},
     journal={Proceedings of the National Academy of Sciences of the United
  States of America},
      volume={43},
      number={8},
       pages={754\ndash 758},
         url={http://www.jstor.org/stable/89599},
}

\bib{RADO}{article}{
      author={Rado, Tibor},
       title={{On Plateau's Problem}},
        date={1930-07},
     journal={The Annals of Mathematics},
      volume={31},
      number={3},
       pages={457},
}

\bib{RIVAC}{article}{
      author={Rivi{\`{e}}re, Tristan},
       title={Everywhere discontinuous harmonic maps into spheres},
        date={1995},
     journal={Acta Mathematica},
      volume={175},
      number={2},
       pages={197\ndash 226},
}

\bib{RIV}{article}{
      author={Rivi{\`{e}}re, Tristan},
       title={Conservation laws for conformally invariant variational
  problems},
        date={2006-12},
     journal={Inventiones mathematicae},
      volume={168},
      number={1},
       pages={1\ndash 22},
}

\bib{SAP}{article}{
      author={{\v S}ver{\'a}k, Vladimir},
       title={On {T}artar's conjecture},
        date={1993},
        ISSN={0294-1449},
     journal={Annales de l'Institut Henri Poincar\'{e}. Analyse Non
  Lin\'{e}aire},
      volume={10},
      number={4},
       pages={405\ndash 412},
      review={\MR{1246459}},
}

\bib{LSP}{article}{
      author={Sz{\'{e}}kelyhidi, L{\'{a}}szl{\'{o}}},
       title={The {R}egularity of {C}ritical {P}oints of {P}olyconvex
  {F}unctionals},
        date={2004-01},
        ISSN={0003-9527},
     journal={Archive for Rational Mechanics and Analysis},
      volume={172},
      number={1},
       pages={133\ndash 152},
}

\bib{RT}{article}{
      author={Tione, Riccardo},
       title={Minimal graphs and differential inclusions},
        date={2021-02},
     journal={Communications in Partial Differential Equations},
      volume={46},
      number={6},
       pages={1162\ndash 1194},
}

\bib{WANG}{article}{
      author={Wang, Mu-Tao},
       title={{The Dirichlet problem for the minimal surface system in
  arbitrary dimensions and codimensions}},
        date={2003},
     journal={Communications on Pure and Applied Mathematics},
      volume={57},
      number={2},
       pages={267\ndash 281},
}

\bib{WAN1}{article}{
      author={Wang, Mu-Tao},
       title={Interior gradient bounds for solutions to the minimal surface
  system},
        date={2004},
     journal={American Journal of Mathematics},
      volume={126},
      number={4},
       pages={921\ndash 934},
}

\end{biblist}
\end{bibdiv}
\end{document}